\documentclass[12pt]{amsart}

\usepackage{amsmath}
\usepackage{amssymb}
\usepackage{amsthm}
\usepackage{pstricks}
\usepackage{pst-node}
\usepackage{pst-plot}
\usepackage{subfigure}
\usepackage{cite}
\usepackage{verbatim}

\numberwithin{equation}{section}

\newtheorem{thm}{Theorem}[section]
\newtheorem{prop}[thm]{Proposition}

\newtheorem{cor}[thm]{Corollary}

\theoremstyle{definition}
\newtheorem{defin}[thm]{Definition}

\newtheorem{ex}[thm]{Example}

\newcommand{\meet}[2]{#1 \cap #2}
\newcommand{\join}[2]{\langle #1, #2 \rangle}
\newcommand{\conv}{\textrm{conv}}
\newcommand{\trace}{\textrm{trace}}

\begin{document}
\title{The limit point of the pentagram map}
\author{Max Glick}
\address{Department of Mathematics, University of Connecticut, Storrs, CT 06269, USA}
\thanks{Partially supported by NSF grant DMS-1303482}

\begin{abstract}
The pentagram map is a discrete dynamical system defined on the space of polygons in the plane. In the first paper on the subject, R. Schwartz proved that the pentagram map produces from each convex polygon a sequence of successively smaller polygons that converges exponentially to a point. We investigate the limit point itself, giving an explicit description of its Cartesian coordinates as roots of certain degree three polynomials.
\end{abstract}

\maketitle

\section{Introduction}
The pentagram map is a discrete dynamical system defined on the space of polygons in the plane.  Figure \ref{figT} shows an instance of the pentagram map, denoted $T$, acting on a polygon $A$ and producing another polygon $B$.  Each vertex $B_i$ of $B$ is constructed as the intersection of two shortest diagonals of $A$, namely $\overleftrightarrow{A_{i-1}A_{i+1}}$ and $\overleftrightarrow{A_iA_{i+2}}$.  Note that if $A$ is convex then $B$ will also be convex and will lie in the interior of $A$.

\begin{figure}[h]
\begin{pspicture}(0,.5)(6,5)
\pspolygon(1,2)(1,3)(2,4.5)(3,5)(4.5,4.5)(5,3.5)(4.5,2)(3.5,1)(2,1)
  \pspolygon[linestyle=dashed](1,2)(2,4.5)(4.5,4.5)(4.5,2)(2,1)(1,3)(3,5)(5,3.5)(3.5,1)
  \pspolygon(1.22,2.55)(1.67,3.67)(2.5,4.5)(3.67,4.5)(4.5,3.87)(4.5,2.67)(3.97,1.79)(2.75,1.3)(1.62,1.75)
	\uput[ul](2,4.5){$A$}
	\uput[dr](1.67,3.67){$B=T(A)$}
\end{pspicture}
\caption{An application of the pentagram map}
\label{figT}
\end{figure}
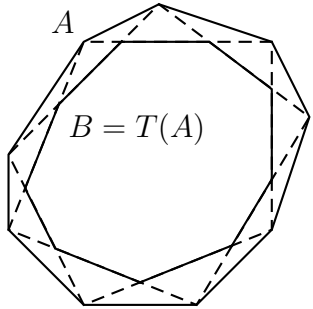

The modern study of the pentagram map was initiated by R. Schwartz in 1992 and his first main result \cite[Theorem 3.1]{S1} was that if $A$ is convex then the sequence of polygons $T^k(A)$ for $k=0,1,2,3,\ldots$ converges exponentially to a single point
$(X,Y) \in \mathbb{R}^2$.  One of the open problems in that paper asked if $X$ and $Y$ are analytic functions of the coordinates of the vertices of $A$.

The pentagram map has seen a spike in popularity in the current decade thanks largely to the discovery that it is a discrete integrable system \cite{OST1,OST2,So,GSTV}, and also because of emerging connections with cluster algebras \cite{G, GSTV}.  In a sense, the recent work differs significantly from the first paper \cite{S1} in that
\begin{enumerate}
\item for the purposes of integrability and cluster algebras, it is more natural to have the pentagram map act not on individual polygons but on projective equivalence classes of polygons, and
\item there has been a focus on generalized pentagram maps \cite{KS1,KS2,KS3,M1,M2,GP}, which are not known to possess a property analogous to preserving convexity.
\end{enumerate}

The present paper returns to the matter of the limit point $(X,Y)$ of the pentagram map acting on a convex polygon $A$.  The main result is that $X$ and $Y$ are not just analytic functions of the coordinates of the vertices of $A$, but are in fact algebraic.

\begin{thm} \label{thmMain}
Let $(x_1,y_1),\ldots, (x_n,y_n)$ be vertices of a convex $n$-gon $A$ and let 
\begin{displaymath}
(X,Y) = \lim_{k \to \infty} T^k(A).
\end{displaymath}  
Then there is a field extension of $\mathbb{Q}(x_1,y_1,\ldots, x_n,y_n)$ of degree at most $3$ containing both $X$ and $Y$.
\end{thm}

The proof is constructive in the sense that it provides a direct method to calculate $X$ and $Y$.  First, lift the vertices of $A$ to vectors
\begin{displaymath}
u_i = \left[\begin{array}{c} x_i \\ y_i \\ 1\end{array}\right]
\end{displaymath}
in $\mathbb{R}^3$.  Define a function $L_A:\mathbb{R}^3 \to \mathbb{R}^3$ by
\begin{equation} \label{eqL0}
L_A(v) = nv -\sum_{j=1}^n \frac{|u_{j-1}, v, u_{j+1}|}{|u_{j-1}, u_j, u_{j+1}|}u_j
\end{equation}
where $|\cdot, \cdot, \cdot|$ denotes the determinant of three vectors and all indices are taken modulo $n$.  It is easy to see that $L_A$ is linear.  

\begin{prop}
The lift $[X\ Y\ 1]^T$ of $(X,Y)$ is an eigenvector of $L_A$.
\end{prop}

As a $3 \times 3$ matrix, $L_A$ has entries in $\mathbb{Q}(x_1,y_1,\ldots, x_n,y_n)$, so the extension alluded to in Theorem \ref{thmMain} is formed by adjoining the appropriate eigenvalue.  At that point $X$ and $Y$ can be found in the extension field by solving a linear system.

\begin{ex} \label{exLimit}
Consider the convex heptagon $A$ with vertices $(2,0)$, $(3,1)$, $(3,2)$, $(2,3)$, $(1,3)$, $(0,2)$, $(0,1)$.  Applying the pentagram map five times (see Figure \ref{figLimit}) provides bounds $1.2 < X < 2.0$ and $1.6 < Y < 2.0$ on the limit point.  The formula for $L_A$ given in the next section can be used to calculate
\begin{displaymath}
L_A = \left[ \begin{array}{ccc} 
-6 & -4 & 49 \\
-1 & -7 & 51 \\
-1 & -3 & 27 \\
\end{array} \right]
\end{displaymath}
which has characteristic polynomial $\lambda^3 - 14\lambda^2 - 111\lambda -116$.  An eigenvector $[X\ Y\ 1]^T$ must satisfy
\begin{align*}
49 &= (6+\lambda)X + 4Y \\
51 &= X + (7 + \lambda)Y \\
27 - \lambda &= X + 3Y
\end{align*}
for $\lambda$ an eigenvalue.  The first two equations suffice to calculate
\begin{displaymath}
(X,Y) = \left(\frac{49 \lambda + 139}{\lambda^2 + 13\lambda + 38}, \frac{51 \lambda + 257}{\lambda^2 + 13\lambda + 38}\right).
\end{displaymath}
The eigenvalues are roughly $\lambda \approx -4.613, -1.265, 19.878$ and the third of these gives rise to the limit point
\begin{displaymath}
(X,Y) \approx (1.609, 1.838).
\end{displaymath}
\end{ex}

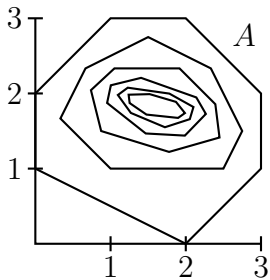
\begin{figure}
\begin{pspicture}(0,-.5)(3,3)
\psline(3,0)(0,0)(0,3)
\psline(1,-.1)(1,.1)
\psline(2,-.1)(2,.1)
\psline(3,-.1)(3,.1)
\psline(-.1,1)(.1,1)
\psline(-.1,2)(.1,2)
\psline(-.1,3)(.1,3)
\uput[d](1,0){$1$}
\uput[d](2,0){$2$}
\uput[d](3,0){$3$}
\uput[l](0,1){$1$}
\uput[l](0,2){$2$}
\uput[l](0,3){$3$}
\pspolygon(2,0)(3,1)(3,2)(2,3)(1,3)(0,2)(0,1)
\pspolygon(1.0000,1.0000)(2.5000,1.0000)(2.7500,1.5000)(2.3333,2.3333)(1.5000,2.7500)(0.6667,2.3333)(0.3333,1.6667)
\pspolygon(1.7778,1.2222)(2.4483,1.4138)(2.3929,1.8571)(1.9167,2.3333)(1.0513,2.3333)(0.7391,2.0435)(0.8750,1.5000)
\pspolygon(1.4675,1.4675)(1.9878,1.4390)(2.2670,1.7273)(2.1401,1.9469)(1.4057,2.2075)(1.0037,2.1086)(0.9540,1.8736)
\pspolygon(1.7227,1.5504)(2.0534,1.6579)(2.1019,1.8194)(1.7817,1.9979)(1.2286,2.0766)(1.0972,1.9794)(1.2698,1.7408)
\pspolygon(1.4771,1.7189)(1.8975,1.6744)(1.9886,1.7390)(1.8697,1.8878)(1.5198,1.9908)(1.2401,1.9833)(1.2537,1.8721)
\uput[ur](2.5,2.5){$A$}
\end{pspicture}
\caption{The polygon considered in Example \ref{exLimit} and its next five iterates under the pentagram map.}
\label{figLimit}
\end{figure}

The map $L_A$ is new and seems to have importance to the pentagram map beyond the problem of describing the limit point.  We list here the main properties of the map, which will be proven throughout this paper.  In the following $A$ is a generic sequence of $n$ points in $\mathbb{R}^2$, $\conv(A)$ denotes the convex hull of the vertices of $A$, and $L_A$ refers depending on context to either the linear map defined in \eqref{eqL0}, the matrix of this linear map, or the induced projective transformation of $\mathbb{R}^2 \subseteq \mathbb{P}^2$.
\begin{itemize}
\item If $A$ is a pentagon then $(L_A - 3I)(A) = T(A)$, and if $A$ is a hexagon then $(L_A - 3I)(A) = T^2(A)$ where $I$ is the identity matrix.
\item (Theorem \ref{thmConserve}) $L_{T(A)} = L_A$
\item (Proposition \ref{propLShrinks}) If $A$ is convex then $L_A(\conv(A)) \subseteq \conv(A)$.  
\item (Proposition \ref{propAxisAligned}) If $A$ is an axis-aligned $2m$-gon, that is one whose vertices satisfy
\begin{align*}
&x_1 = x_2, x_3 = x_4, \ldots, x_{2m-1}=x_{2m}, \\
&y_2 = y_3, y_4 = y_5, \ldots, y_{2m}=y_1,
\end{align*}
then
\begin{equation} \label{eqAxisAligned}
L_A = \left[\begin{array}{ccc}
m & 0 & x_1 + x_3 + \ldots + x_{2m-1} \\
0 & m & y_2 + y_4 + \ldots + y_{2m} \\
0 & 0 & 2m \end{array} \right].
\end{equation}
\end{itemize}

We now make several remarks regarding the above properties, following the same order they were listed.  The fact that $A$ is projectively equivalent to $T(A)$ (respectively $T^2(A))$ if $A$ is a pentagon (respectively hexagon) is classical.  The claim is simply that $L_A - 3I$ is the matrix for the projective transformation realizing this equivalence.  We omit the proof which is purely computational.  It is clear that the limit point must be fixed by this transformation (hence also by $L_A$), so the result of Theorem \ref{thmMain} is only really surprising for $n \geq 7$. 

Because $L_{T(A)} = L_A$, the nine entries $\phi_{i,j}$ of $L_A$ are conserved quantities of the pentagram map.  They satisfy a relation $\phi_{11} + \phi_{22} + \phi_{33} = 2n$ (Proposition \ref{propTrace}) but seem to otherwise be independent.  Note the individual $\phi_{i,j}$ are not invariant under projective transformations, so they must be different from the standard conserved quantities $O_k$ and $E_k$ (see \cite{OST1}).  However, the coefficients of the characteristic polynomial of $L_A$ are projective invariants (Corollary \ref{corProjInvs}).  As just mentioned, the trace is constant, but it would be interesting to express the other two coefficients in terms of the $O_k$ and $E_k$.

The property $L_A(\conv(A)) \subseteq \conv(A)$ can be thought of as a point of commonality with the pentagram map which also sends a convex polygon into its interior.  Schwartz speculates \cite{S1} that some projective transformation applied repeatedly to $A$ may approximate its pentagram map orbit, giving a direct explanation of the quasiperiodic property \cite{S2}.  Although experiments show that $L_A$ does not fit this bill, we can say that it in some sense goes in the right direction.  

Finally, axis-aligned polygons play a special role in the study of the pentagram map, so it is unsurprising that $L_A$ takes a simple form in this case.  Let $A$ be axis-aligned.  Schwartz \cite[Theorem 1.3]{S3} and the author \cite[Theorem 7.6]{G} showed that after a finite number of steps of the pentagram map the vertices of $A$ collapse to a single point.  Axis-aligned polygons are necessarily not convex, so Theorem \ref{thmMain} does not apply directly, but it is natural to consider this point of collapse as being the analogue of the limit point.  Rewriting \eqref{eqAxisAligned} as a map of the plane yields
\begin{displaymath}
L_A(x,y) = \left(\frac{x+X}{2}, \frac{y+Y}{2}\right)
\end{displaymath}
where $(X,Y)$ is the center of mass of the vertex set of $A$.  This formula leads quickly to a new proof of Z. Yao's theorem \cite{Y} that the center of mass equals the point of collapse.

We close the introduction with a comment on a possible future direction.  As established by Ovsienko, Schwartz and Tabachnikov \cite{OST1}, the pentagram map has a continuous limit given by a certain flow on plane curves modeled by the Boussinesq equation.  It is not hard to imagine that the results of the current paper could be extended from convex polygons to closed convex curves, with the sum in \eqref{eqL0} being replaced by an integral.

\medskip

\textbf{Acknowledgment.} I thank Richard Schwartz for several helpful discussions and for pointing out the potential extension to the continuous limit.

\section{Elementary properties of $L_A$} 
We begin with a formula for the entries of $L_A$.  Let $A$ be a polygon with vertices $(x_1,y_1),\ldots, (x_n,y_n)$.  Let $a_{ij}$ for $i \in \mathbb{Z}/(3\mathbb{Z})$ and $j \in \mathbb{Z}/(n\mathbb{Z})$ denote
\begin{displaymath}
a_{i,j} = \begin{cases}
x_j, & i=1\\
y_j, & i=2\\
1, & i=3
\end{cases}.
\end{displaymath}
Lastly, let $\phi_{i,j}$ for $i,j=1,2,3$ be the entries of $L_A$ viewed as a matrix.
\begin{prop}
\begin{equation} \label{eqphi}
\phi_{i,j} = n\delta_{i,j} - \sum_{k=1}^n \frac{(a_{j-1,k-1}a_{j+1,k+1}-a_{j-1,k+1}a_{j+1,k-1})a_{i,k}}
{\left| \begin{array}{ccc} 
a_{1,k-1} & a_{1,k} & a_{1,k+1} \\
a_{2,k-1} & a_{2,k} & a_{2,k+1} \\
a_{3,k-1} & a_{3,k} & a_{3,k+1}
\end{array}\right|}
\end{equation}
\end{prop}

\begin{proof}
The formula is obtained by plugging $v=e_j$ into \eqref{eqL0} and taking the $i$th entry of the result.
\end{proof}

\begin{prop} \label{propTrace}
For any $n$-gon, $\trace(L_A) = 2n$.
\end{prop}

\begin{proof}
In the expression for $\trace(L_A) = \phi_{1,1} + \phi_{2,2} + \phi_{3,3}$ obtained using \eqref{eqphi}, the $k$th summands add up to $1$ for all $k$.  Hence 
\begin{displaymath}
\trace(L_A) = n+n+n - \sum_{k=1}^n 1 = 2n.
\end{displaymath}
\end{proof}

\begin{prop} Let $\psi \in GL_3(\mathbb{R})$, let $A$ be an $n$-gon, and let $B = \psi(A)$.  Then
\begin{displaymath}
L_B = \psi L_A \psi^{-1}.
\end{displaymath}
\end{prop}

\begin{proof}
Given $v \in \mathbb{R}^3$ we need to show
\begin{displaymath}
L_B(\psi v) = \psi L_A(v).
\end{displaymath}
Note that \eqref{eqL0} is invariant under arbitrary rescaling of the vectors $u_1,\ldots, u_n$.  Hence we can take any lifts $u_1,\ldots, u_n$ of the vertices of $A$ in the calculation of $L_A$ and the corresponding lifts $\psi u_1, \ldots, \psi u_n$ in the calculation of $L_B$.  So 
\begin{align*}
L_B(\psi v) &= n\psi v - \sum_{j=1}^n \frac{|\psi u_{j-1}, \psi v, \psi u_{j+1}|}{|\psi u_{j-1}, \psi u_j, \psi u_{j+1}|} \psi u_j \\
&= n\psi v - \sum_{j=1}^n \frac{|u_{j-1}, v, u_{j+1}|}{|u_{j-1}, u_j, u_{j+1}|} \psi u_j \\
&= \psi L_A(v)
\end{align*}
as desired.
\end{proof}

It follows that $L_A$ and $L_{\psi(A)}$ have the same eigenvalues leading to the following.

\begin{cor} \label{corProjInvs}
The characteristic polynomial of $L_{A}$ is invariant under projective transformations of $A$.
\end{cor}

\section{Conservation under the pentagram map} The purpose of this section is to prove the linear map $L_A$ is conserved by the pentagram map.  It is convenient to work with an abstract three-dimensional vector space $V$ and the corresponding projective plane $P = \mathbb{P}(V)$.  Let $A = (A_1,A_2,\ldots, A_n)$ be an $n$-tuple of points of $P$ in general position (no three points on a common line).  Choose arbitrarily nonzero lifts $u_1,\ldots, u_n \in V$ of $A_1,\ldots, A_n$.  Finally, define a function $L_A:V \to V$ by
\begin{equation} \label{eqL1}
L_A(v) = nv - \sum_{j=1}^n \frac{|u_{j-1}, v, u_{j+1}|}{|u_{j-1}, u_j, u_{j+1}|}u_j.
\end{equation}
Here and throughout, indices are taken modulo $n$.  The notation $|\cdot, \cdot, \cdot|$ refers to a determinantal form on $V$.  The choice of the form does not matter as a ratio of two determinants will always have the same value.  Two other easy observations are
\begin{enumerate}
\item $L_A$ is linear and
\item $L_A$ does not depend on the choice of the lifts $u_j$.
\end{enumerate}
To sum up, we have a rational map
\begin{align*}
P^n &\to End(V) \\
A &\mapsto L_A
\end{align*}

\begin{thm} \label{thmConserve}
For generic $A = (A_1,\ldots, A_n) \in P^n$
\begin{displaymath}
L_{T(A)} = L_A.
\end{displaymath}
\end{thm}

To prove this result, it is easiest to break the pentagram map into two pieces $\alpha_1$ and $\alpha_2$ and consider each piece individually.  To this end, let $V^*$ denote the dual space of $V$ and $P^* = \mathbb{P}(V^*)$.  Given distinct points $A,B \in P$, there is a unique up to scaling, nonzero $f \in V^*$ that vanishes on both $A$ and $B$.  Let $\join{A}{B}$ denote the corresponding point in $P^*$, visualized as the line in $P$ containing $A$ and $B$. 

Define rational maps $\alpha_1,\alpha_2:P^n \to (P^*)^n$ by
\begin{displaymath}
\alpha_1(A) = (\join{A_1}{A_2}, \join{A_2}{A_3}, \join{A_3}{A_4}, \ldots, \join{A_n}{A_1})
\end{displaymath} 
and
\begin{displaymath}
\alpha_2(A) = (\join{A_n}{A_2}, \join{A_1}{A_3}, \join{A_2}{A_4}, \ldots, \join{A_{n-1}}{A_1}).
\end{displaymath}
There is the usual identification of $(V^*)^*$ with $V$ inducing an identification of $(P^*)^*$ with $P$.  Hence starting from $P^*$ we also get maps $\alpha_1,\alpha_2: (P^*)^n \to P^n$.  We then have that, up to reindexing vertices, $\alpha_1 \circ \alpha_1$ and $\alpha_2 \circ \alpha_2$ are the identity, $\alpha_1 \circ \alpha_2$ is the pentagram map, and $\alpha_2 \circ \alpha_1$ is its inverse.  This decomposition of the pentagram map as a product of two involutions is presented in \cite{S3} with greater attention paid to the indexing.

For a linear map $L:V \to W$ let $L^*$ be the dual map $L^*:W^* \to V^*$ defined by $(L^*(f))(v) = f(L(v))$ for all $v \in V$, $f \in W^*$.

\begin{prop} \label{propalpha2}
Let $A = (A_1,\ldots, A_n) \in P^n$.  Then
\begin{displaymath}
L_{\alpha_2(A)} = L_A^*.
\end{displaymath}
\end{prop}

\begin{proof}
Let $u_1,\ldots, u_n$ be lifts of $A_1,\ldots, A_n$ as before.  Then
\begin{align*}
L_A(v) &= nv - \sum_{j=1}^n \frac{|u_{j-1}, v, u_{j+1}|}{|u_{j-1}, u_j, u_{j+1}|}u_j \\
&= nv - \sum_{j=1}^n \frac{f_j(v)}{f_j(u_j)} u_j
\end{align*}
where $f_j = |u_{j-1}, \cdot, u_{j+1}| \in V^*$.  Now $f_j$ vanishes on both $A_{j-1}$ and $A_{j+1}$ so it is a lift of $\join{A_{j-1}}{A_{j+1}}$.  Hence
\begin{displaymath}
L_{\alpha_2(A)}(g) = ng - \sum_{j=1}^n \frac{|f_{j-1}, g, f_{j+1}|}{|f_{j-1}, f_j, f_{j+1}|}f_j
\end{displaymath}
for $g \in V^*$.  Note that $|f_{j-1}, \cdot, f_{j+1}|$ and ``evaluation at $u_j$'' are both functionals on $V^*$ that vanish at $f_{j-1}$ and $f_{j+1}$.  Hence they are scalar multiples of each other.  It follows that
\begin{displaymath}
L_{\alpha_2(A)}(g) = ng - \sum_{j=1}^n \frac{g(u_j)}{f_j(u_j)}f_j.
\end{displaymath}
Therefore
\begin{displaymath}
(L_{\alpha_2(A)}(g))(v) = ng(v) - \sum_{j=1}^n \frac{g(u_j)f_j(v)}{f_j(u_j)} = g(L_A(v))
\end{displaymath}
for all $v \in V$ and $g \in V^*$ as desired.
\end{proof}

There is an alternate formula for $L_A$ which is better suited for dealing with $\alpha_1$.  By Cramer's rule
\begin{displaymath}
\frac{|v, u_j, u_{j+1}|u_{j-1} + |u_{j-1}, v, u_{j+1}|u_j + |u_{j-1},u_j,v|u_{j+1}}{|u_{j-1}, u_j, u_{j+1}|} = v.
\end{displaymath}
Therefore
\begin{equation} \label{eqL2}
L_A(v) = \sum_{j=1}^n \left(\frac{|v, u_j, u_{j+1}|}{|u_{j-1}, u_j, u_{j+1}|}u_{j-1} + \frac{|u_{j-1},u_j,v|}{|u_{j-1}, u_j, u_{j+1}|}u_{j+1}\right)
\end{equation}

\begin{prop} \label{propalpha1}
Let $A = (A_1,\ldots, A_n) \in P^n$.  Then
\begin{displaymath}
L_{\alpha_1(A)} = L_A^*.
\end{displaymath}
\end{prop}

\begin{proof}
Let $f_j = |\cdot, u_j, u_{j+1}| = |u_j, u_{j+1}, \cdot| \in V^*$.  Then
\begin{displaymath}
L_A(v) = \sum_{j=1}^n \left(\frac{f_j(v)}{f_j(u_{j-1})}u_{j-1} + \frac{f_{j-1}(v)}{f_{j-1}(u_{j+1})}u_{j+1}\right).
\end{displaymath}
On the other hand, $f_j$ is a lift of $\join{A_j}{A_{j+1}}$ for all $j$ so
\begin{align*}
L_{\alpha_1(A)}(g) &= \sum_{j=1}^n \left(\frac{|g, f_j, f_{j+1}|}{|f_{j-1},f_j,f_{j+1}|}f_{j-1} + \frac{|f_{j-1},f_j,g|}{|f_{j-1},f_j,f_{j+1}|}f_{j+1}\right) \\
&= \sum_{j=1}^n \left(\frac{g(u_{j+1})}{f_{j-1}(u_{j+1})}f_{j-1} + \frac{g(u_j)}{f_{j+1}(u_j)}f_{j+1}\right)
\end{align*}
by similar reasoning as in the proof of Proposition \ref{propalpha2}.  Putting everything together
\begin{align*}
g(L_A(v)) &= \sum_{j=1}^n \left(\frac{f_j(v)g(u_{j-1})}{f_j(u_{j-1})} + \frac{f_{j-1}(v)g(u_{j+1})}{f_{j-1}(u_{j+1})}\right) \\
&= \sum_{j=1}^n \left(\frac{f_{j-1}(v)g(u_{j+1})}{f_{j-1}(u_{j+1})} + \frac{f_{j+1}(v)g(u_j)}{f_{j+1}(u_j)}\right) \\
&= (L_{\alpha_1(A)}(g))(v)
\end{align*}
\end{proof}

Combining Propositions \ref{propalpha2} and \ref{propalpha1} yields
\begin{displaymath}
L_{T(A)} = L_{\alpha_1(\alpha_2(A))} = (L_A^*)^* = L_A
\end{displaymath}
completing the proof of Theorem \ref{thmConserve}.

\section{Proof of main theorem}
In this section we prove Theorem \ref{thmMain}.  Let $A$ be a convex $n$-gon.  The main idea is to show that the limit point of $A$ corresponds to an eigenvector of $L_A$, which follows easily from Theorem \ref{thmConserve} together with the following.

\begin{prop} \label{propLShrinks}
If $Q \in \conv(A)$ then $L_A(Q) \in \conv(A)$.
\end{prop}

\begin{proof}
Let $v,u_1,\ldots, u_n \in \mathbb{R}^3$ be lifts of $Q, A_1,\ldots, A_n$ respectively, choosing all lifts on the hyperplane $z=1$.  Then by \eqref{eqL2}
\begin{displaymath}
L_A(v) = \sum_{j=1}^n \left(\frac{|u_{j-2},u_{j-1},v|}{|u_{j-2}, u_{j-1}, u_j|} + \frac{|v, u_{j+1}, u_{j+2}|}{|u_j, u_{j+1}, u_{j+2}|}\right)u_j.
\end{displaymath}
Geometrically, the coefficient 
\begin{displaymath}
\frac{|u_{j-2},u_{j-1},v|}{|u_{j-2}, u_{j-1}, u_j|}
\end{displaymath}
equals the ratio of the areas of $\triangle A_{j-2}A_{j-1}Q$ and $\triangle A_{j-2}A_{j-1}A_j$.  Since $Q \in \conv(A)$, $Q$ lies (weakly) on the same side of $\overleftrightarrow{A_{j-2}A_{j-1}}$ as $A_j$ does so the ratio is taken with a positive sign.  For similar reasons
\begin{displaymath}
\frac{|v, u_{j+1}, u_{j+2}|}{|u_j, u_{j+1}, u_{j+2}|} \geq 0.
\end{displaymath}
The total coefficient of $u_j$ is strictly positive since $\overleftrightarrow{A_{j-2}A_{j-1}}$ and $\overleftrightarrow{A_{j+1}A_{j+2}}$ do not intersect in $\conv(A)$.  Therefore $L_A(v)$ is a positive linear combination of $u_1,u_2,\ldots,u_n$.  Scaling down to $z=1$ yields that $L_A(Q)$ is a convex combination of $A_1,\ldots, A_n$.  
\end{proof}

\begin{prop} \label{propEigenvector}
Let $A$ be a convex polygon and $(X,Y) = \lim_{k \to \infty}T^k(A)$.  Then $[X\ Y\ 1]^T$ is an eigenvector of $L_A$, and the associated eigenspace is one-dimensional.
\end{prop}

\begin{proof}
By Proposition \ref{propLShrinks}, $L_A$ restricts to a continuous map from $\conv(A)$ to itself.  For each $k \geq 0$ we have $(X,Y) \in \conv(T^k(A))$ so
\begin{align*}
L_A(X,Y) &= L_{T^k(A)}(X,Y) \quad \textrm{(by Theorem \ref{thmConserve})} \\
&\in \conv(T^k(A)) \quad \textrm{(by Proposition \ref{propLShrinks})}
\end{align*}
It follows that $L_A(X,Y)$ equals the limit point $(X,Y)$.  Lifting to $\mathbb{R}^3$, $[X\ Y\ 1]^T$ must be an eigenvector of $L_A$.

Suppose for the sake of contradiction that $[X\ Y\ 1]^T$ is part of a larger dimensional eigenspace.  Projecting to the plane $z=1$ gives a line containing $(X,Y)$ with all its points fixed by $L_A$.  This line must intersect a side or vertex of $A$, say at the point $Q$.  Then $L_A(Q) = Q$ which is a contradiction as the proof of Proposition \ref{propLShrinks} shows that $L_A(Q)$ lies in the interior of $A$.
\end{proof}

The problem of determining the limit point is now reduced to linear algebra.

\begin{proof}[Proof of Theorem \ref{thmMain}]
Let $A$ be a convex polygon with vertices $(x_1,y_1),\ldots, (x_n,y_n)$.  Then $L_A$ is a $3$-by-$3$ matrix whose entries are rational functions of the $x_j$ and $y_j$.  By Proposition \ref{propEigenvector}, there is an eigenvalue $\lambda$ of $L_A$ with geometric multiplicity $1$ for which $[X\ Y\ 1]^T$ is an eigenvector.  Hence $(X,Y)$ is the unique solution to a linear system
\begin{displaymath}
L_A\left(\left[ \begin{array}{c} X \\ Y \\ 1 \\ \end{array} \right]\right)
= \left[ \begin{array}{c} \lambda X \\ \lambda Y \\ \lambda \\ \end{array} \right].
\end{displaymath}
Row reduction produces the solution with $X,Y \in \mathbb{Q}(x_1,y_1,\ldots, x_n,y_n, \lambda)$.
\end{proof}

\section{Axis-aligned polygons}
We now apply the results of the previous sections to the special case of axis aligned polygons.  Let $A$ be a $2m$-gon with vertices
\begin{displaymath}
(x_1,y_{2m}), (x_1,y_2), (x_3,y_2), (x_3, y_4), (x_5, y_4), \ldots, (x_{2m-1}, y_{2m}).
\end{displaymath}

\begin{prop} \label{propAxisAligned}
For $A$ as above, $L_A: \mathbb{R}^3 \to \mathbb{R}^3$ is given in matrix form by
\begin{displaymath}
L_A = \left[\begin{array}{ccc}
m & 0 & x_1 + x_3 + \ldots + x_{2m-1} \\
0 & m & y_2 + y_4 + \ldots + y_{2m} \\
0 & 0 & 2m \end{array} \right].
\end{displaymath}
\end{prop}
\begin{proof}
Lift the vertices of $A$ to $\mathbb{R}^3$ as 
\begin{displaymath}
u_{2i} = \left[\begin{array}{c} x_{2i-1} \\ y_{2i} \\ 1 \end{array}\right] \quad
u_{2i+1} = \left[\begin{array}{c} x_{2i+1} \\ y_{2i} \\ 1 \end{array}\right].
\end{displaymath}
If $v = [x \ y \ z]^T$ then by direct calculations
\begin{align*}
|u_{2i-1}, u_{2i}, v| &= (y_{2i}-y_{2i-2})(x_{2i-1}z-x) \\
|u_{2i-1}, u_{2i}, u_{2i+1}| &= (y_{2i}-y_{2i-2})(x_{2i-1}-x_{2i+1}) \\
|u_{2i}, u_{2i+1}, v| &= (x_{2i-1}-x_{2i+1})(y_{2i}z-y) \\
|u_{2i}, u_{2i+1}, u_{2i+2}| &= (x_{2i-1}-x_{2i+1})(y_{2i}-y_{2i+2})
\end{align*}

Reorganizing \eqref{eqL2} yields
\begin{displaymath}
L_A(v) = \sum_{j=1}^{2m} \left(\frac{|v, u_{j+1}, u_{j+2}|}{|u_j, u_{j+1}, u_{j+2}|}u_j + \frac{|u_{j-1},u_j,v|}{|u_{j-1}, u_j, u_{j+1}|}u_{j+1}\right).
\end{displaymath}
If $j=2i-1$ then
\begin{align*}
&\frac{|v, u_{j+1}, u_{j+2}|}{|u_j, u_{j+1}, u_{j+2}|}u_j + \frac{|u_{j-1},u_j,v|}{|u_{j-1}, u_j, u_{j+1}|}u_{j+1} \\
&= \frac{y_{2i}z-y}{y_{2i}-y_{2i-2}} \left[ \begin{array}{c} x_{2i-1} \\ y_{2i-2} \\ 1 \end{array}\right]
+ \frac{y_{2i-2}z-y}{y_{2i-2}-y_{2i}} \left[ \begin{array}{c} x_{2i-1} \\ y_{2i} \\ 1 \end{array}\right] \\
&= y\left[ \begin{array}{c} 0 \\ 1 \\ 0 \end{array} \right] + z\left[ \begin{array}{c} x_{2i-1} \\ 0 \\ 1 \end{array} \right]
\end{align*}
while if $j=2i$ then
\begin{align*}
&\frac{|v, u_{j+1}, u_{j+2}|}{|u_j, u_{j+1}, u_{j+2}|}u_j + \frac{|u_{j-1},u_j,v|}{|u_{j-1}, u_j, u_{j+1}|}u_{j+1} \\
&= \frac{x_{2i+1}z-x}{x_{2i+1}-x_{2i-1}} \left[ \begin{array}{c} x_{2i-1} \\ y_{2i} \\ 1 \end{array}\right]
+ \frac{x_{2i-1}z-x}{x_{2i-1}-x_{2i+1}} \left[ \begin{array}{c} x_{2i+1} \\ y_{2i} \\ 1 \end{array}\right] \\
&= x\left[ \begin{array}{c} 1 \\ 0 \\ 0 \end{array} \right] + z\left[ \begin{array}{c} 0 \\ y_{2i} \\ 1 \end{array} \right]
\end{align*}
Summing over all $j$,
\begin{displaymath}
L_A(v) = \left[\begin{array}{c}
mx + (x_1 + x_3 + \ldots + x_{2m-1})z \\
my + (y_2 + y_4 + \ldots + y_{2m})z \\
2mz \end{array}\right]
\end{displaymath}
as desired.
\end{proof}

As demonstrated by Schwartz \cite{S3}, iteration of the pentagram map on an axis-aligned polygon can be modeled by Dodgson's condensation method of computing determinants.  An application of this idea is a remarkable incidence theorem \cite[Theorem 1.3]{S3} that if $A$ is an axis-aligned $2m$-gon then $T^{m-2}(A)$ has its odd vertices lying on one line and its even vertices on another.  Equally remarkably, Yao \cite[Theorem 1.3]{Y} demonstrated that the intersection point of these two lines, termed the point of collapse, is the center of mass of the vertices of the original polygon $A$.  We now have a new proof of Yao's theorem.

\begin{thm}[Yao]
Let $A$ be an axis-aligned $2m$-gon with vertices as before and let $B = T^{m-2}(A)$.  Let $l_1$ be the line containing $B_1,B_3,\ldots, B_{2m-1}$ and $l_2$ the line containing $B_2,B_4,\ldots, B_{2m}$.  Then $l_1$ and $l_2$ intersect at the point
\begin{displaymath}
\left(\frac{x_1+x_3 +\ldots + x_{2m-1}}{m}, \frac{y_2 + y_4 + \ldots + y_{2m}}{m}\right).
\end{displaymath}
\end{thm}

\begin{proof}
We know
\begin{displaymath}
L_B = L_A = \left[\begin{array}{ccc}
m & 0 & x_1 + x_3 + \ldots + x_{2m-1} \\
0 & m & y_2 + y_4 + \ldots + y_{2m} \\
0 & 0 & 2m \end{array} \right]
\end{displaymath}
The eigenspace for $\lambda=2m$ has dimension one and is spanned by
\begin{displaymath}
\left[\begin{array}{c}
x_1 + x_3 + \ldots x_{2m-1} \\
y_2 + y_4 + \ldots y_{2m} \\
m \end{array}\right].
\end{displaymath}
On the other hand, let $Q = \meet{l_1}{l_2}$.  Lift $Q$ to $v \in \mathbb{R}^3$ and lift $B_j$ to $u_j$.  For each $j$, we have that $B_{j-1}$, $Q$, and $B_{j+1}$ are collinear so $|u_{j-1}, v, u_{j+1}| = 0$.  By \eqref{eqL1}
\begin{displaymath}
L_B(v) = 2mv.
\end{displaymath}
So, $v$ equals up to scale the above eigenvector and
\begin{displaymath}
Q = \left(\frac{x_1+x_3 +\ldots + x_{2m-1}}{m}, \frac{y_2 + y_4 + \ldots + y_{2m}}{m}\right).
\end{displaymath}
\end{proof}

\end{document}